\documentclass[12pt]{article}
\usepackage{amssymb, amsmath}
\usepackage{fullpage,subfigure,times,comment}

\usepackage{graphicx,amsmath,amsfonts,amscd,amssymb,bm,cite,epsfig,epsf,url,color,algorithm,algorithmic}
\usepackage{fullpage} \usepackage[small,bf]{caption}
\newtheorem{theorem}{Theorem}[section]
\newtheorem{lemma}[theorem]{Lemma}
\newtheorem{corollary}[theorem]{Corollary}
\newtheorem{proposition}[theorem]{Proposition}
\newtheorem{definition}[theorem]{Definition}

\newtheorem{remark}[subsection]{Remark}

\def \endprf{\hfill {\vrule height6pt width6pt depth0pt}\medskip}
\newenvironment{proof}{\noindent {\bf Proof} }{\endprf\par}

\newcommand{\ts}[1]{\textcolor{red}{#1}}

\newcommand{\R}{\mathbb{R}}
\newcommand{\C}{\mathbb{C}}
\newcommand{\N}{\mathbb{N}}

\newcommand{\Z}{\mathbb{Z}}

\newcommand{\herm}{{\varphi}}
\newcommand{\hermm}{{\varphi}_{m}}
\newcommand{\hermmn}{{\varphi}_{m,n}}

\def\phi{{\varphi}}

\newcommand{\bigO}{\mathcal{O}}

\newcommand{\Hopt}{{H_{\alpha,\beta,\theta}}}
\newcommand{\Topt}{{H_{\alpha,\beta,\theta}}}
\newcommand{\Hopnt}{{P^{(n)}_{\alpha,\beta,\theta}}}
\newcommand{\Hop}{{H_{\alpha,\beta}}}
\newcommand{\Hopn}{{P^{(n)}_{\alpha,\beta}}}

\newcommand{\Topnt}{{H^{(n)}_{\alpha,\beta,\theta}}}

\numberwithin{equation}{section}

\begin{document}                        

\title{\bf Almost Eigenvalues and Eigenvectors \\ of Almost Mathieu Operators}

\author{Thomas Strohmer and Timothy Wertz\\
Department of Mathematics \\
University of California, Davis  \\
Davis CA 95616 \\
\{strohmer,tmwertz\}@math.ucdavis.edu}

\date{}

\maketitle

\begin{abstract}
The {\em almost Mathieu operator} is the discrete Schr\"odinger operator $H_{\alpha,\beta,\theta}$ on $\ell^2(\Z)$
defined via  $(H_{\alpha,\beta,\theta}f)(k) = f(k + 1) + f(k - 1) + \beta \cos(2\pi \alpha k + \theta) f(k)$.
We derive explicit estimates for the eigenvalues at the edge of the spectrum of the finite-dimensional almost Mathieu operator $\Topnt$. We furthermore show that the  (properly rescaled) $m$-th Hermite function $\phi_m$ is an approximate
eigenvector of $\Topnt$, and that it satisfies the same properties that characterize the true eigenvector associated to the $m$-th largest eigenvalue of $\Topnt$. Moreover, a properly translated and modulated version of $\phi_m$
 is also an approximate eigenvector of $\Topnt$, and it satisfies the properties that characterize the true eigenvector associated to the $m$-th largest (in modulus) negative eigenvalue. The results hold at the edge of the spectrum, for any choice of $\theta$ and under very mild conditions on $\alpha$ and $\beta$. We also give precise estimates for the size of the ``edge'', and extend some of our results to $\Hopt$. The ingredients for our proofs comprise Taylor expansions, basic time-frequency analysis, Sturm sequences, and perturbation theory for eigenvalues and eigenvectors.
Numerical simulations demonstrate the tight fit of the theoretical estimates.
\end{abstract}


\section{Introduction}
\label{s:intro}

We consider the {\em almost Mathieu operator}  $H_{\alpha,\beta,\theta}$ on $\ell^2(\Z)$, given by
\begin{equation}
(H_{\alpha,\beta,\theta}f)(x) = f(x + 1) + f(x - 1) + 2 \beta \cos(2\pi \alpha k + \theta) f(x),
\label{mathieu}
\end{equation}
with $x\in \Z$, $\beta \in \R, \alpha \in [-\frac{1}{2}, \frac{1}{2})$, and $\theta \in [0,2\pi)$.
This operator is interesting both from a phyiscal and a mathematical point of view~\cite{Shubin94,Last95}.
In physics, for instance, it serves as a model for Bloch electrons in a magnetic field~\cite{Hof76}. In mathematics, it appears in connection with graph theory and random walks on the Heisenberg group~\cite{DS94,BVZ97} and rotation algebras~\cite{Boca08}.
A major part of the mathematical fascination of almost Mathieu operators stems from their interesting spectral properties, 
obtained by varying the parameters $\alpha, \beta, \theta$, which has led to some deep and beautiful mathematics,
see e.g~\cite{BS82,AMS90,Jito95,LM07,AD08}. For example, it is known that the spectrum of the almost Mathieu operator is a Cantor set 
for all irrational $\alpha$ and for all $\beta \neq 0$, cf.~\cite{AJ09}.
Furthermore, if $\beta >1$ then $\Hopt$ exhibits Anderson localization, i.e.,  
the spectrum is pure point with exponentially decaying eigenvectors~\cite{Jito94}.

A vast amount of literature exists devoted to the study of the bulk spectrum of $\Topt$ and its 
structural characteristics, but very little seems be known about the edge of the spectrum. For instance,
what is the size of the extreme eigenvalues of $\Topt$, how do they depend on $\alpha,\beta,\theta$
and what do the associated eigenvectors look like? These are exactly the questions we will address
in this paper.

While in general the localization of the eigenvectors of $\Topt$ depends on the choice of $\beta$, it turns
out that there exist approximate eigenvectors associated with the extreme eigenvalues of $\Topt$ which 
are {\em always} exponentially localized. Indeed, we will show that for small $\alpha$ the $m$-th Hermitian function $\phi_m$ as well as 
certain translations and modulations of $\phi_m$ form {\em almost} eigenvectors of $\Topt$  regardless whether $\alpha$ is rational or
irrational and as long as the product $\alpha \sqrt{\beta}$ is small.  

There is a natural heuristic explanation why Hermitian functions emerge in connection with almost Mathieu operators.
Consider the continuous-time version of $\Hopt$ in~\eqref{mathieu} by letting $x\in\R$ and set $\alpha=1, \beta=1, \theta=0$. Then $\Hopt$ 
commutes with the Fourier transform on $L^2(\R)$. It is well-known that Hermite functions are eigenfunctions of the Fourier transform, ergo 
Hermite functions are eigenvectors of the aforementioned continuous-time analog of the Mathieu operator. Of course,
it is no longer true that the discrete $\Topt$ commutes with the corresponding Fourier transform (nor do we want to restrict ourselves to one specific choice 
of $\alpha$ and $\beta$).  But nevertheless it may still be true that discretized (and perhaps truncated) Hermite functions are almost
eigenvectors for $\Hopt$. We will see that this is indeed the case under some mild conditions, but it only holds for the first $m$ Hermite
functions where the size of $m$ is either $\bigO(1)$ or $\bigO(1/\sqrt{\gamma})$, where $\gamma = \pi \alpha\sqrt{\beta}$, depending on the desired
accuracy ($\gamma^2$ and $\gamma$, respectively) of the approximation.  We will also
show a certain symmetry for the eigenvalues of $\Hopt$ and use this fact to conclude that  a properly translated and modulated Hermite function is an 
approximate eigenvector for the $m$-th largest (in modulus) negative eigenvalue. 

The only other work we are aware of that analyzes the eigenvalues of the almost Mathieu operator at the
edge of the spectrum is~\cite{WPR87}. There, the authors analyze a 
continuous-time model to obtain eigenvalue estimates of the discrete-time operator $\Hopt$. 
They consider the case $\beta = 1,\theta =0$ and small $\alpha$ and arrive at an estimate for the eigenvalues at the right edge of the spectrum
that is not far from our expression for this particular case (after translating their notation into ours and correcting 
what seems to be a typo in~\cite{WPR87}). But there are several differences to our work. First,~\cite{WPR87}
does not provide any results about the eigenvectors of $\Hopt$. Second, \cite{WPR87} does not derive any error estimates for their 
approximation, and indeed, an analysis
of their approach yields that their approximation is only accurate up to order $\gamma$ and not $\gamma^2$.
Third, \cite{WPR87} contains no quantitative characterization of the size of the edge of the spectrum. On the other hand, the scope of~\cite{WPR87} is different from ours.

\medskip

The remainder of the paper is organized as follows. In Subsection~\ref{ss:def} we introduction some notation and definitions used throughout the paper.
In Section~\ref{s:finite} we  derive eigenvector and eigenvalue estimates for the finite-dimensional model of the almost Mathieu operator.
The ingredients for our proof comprise Taylor expansions, basic time-frequency analysis, Sturm sequences, and perturbation theory for eigenvalues
and eigenvectors.
The extension of our main results to the infinite dimensional almost Mathieu operator is carried out in Section~\ref{s:infinite}.
Finally, in Section~\ref{s:numerics} we complement our theoretical findings with numerical simulations.


\subsection{Definitions and Notation}
\label{ss:def}

We define the unitary operators of translation and modulation, denoted $T_{a}$ and $M_{b}$ respectively by
$$ \left( T_a f \right)(x) := f(x-a)  \hspace{2pc} \mbox{and} \hspace{2pc} \left( M_bf \right)(x) := e^{2\pi i bx}f(x),$$
where the translation is understood in a periodic sense if $f$ is a vector of finite length. It will be clear from the context
if we are dealing with finite or infinite-dimensional versions of $T_a$ and $M_b$.
Recall the commutation relations (see e.g. Section 1.2 in~\cite{Gro01})
\begin{equation}
T_a M_b = e^{-2 \pi i a b} M_b T_a.
\label{comm}
\end{equation}

The discrete and periodic Hermite functions we will be using are derived from the standard Hermite functions defined on $\R$ (see e.g.~\cite{ASbook}) by simple discretization and truncation. We do choose a slightly different normalization than in~\cite{ASbook} by introducing
the scaling terms $(\sqrt{2 \gamma})^{2l}$ and  $(\sqrt{2 \gamma})^{2l+1}$, respectively.
\begin{definition}
\label{hermdef}
The scaled Hermite functions $\hermm$ with parameter $\gamma>0$ are
\begin{equation}
\text{for even $m$:} \quad
\phi_{m}(x) =
 e^{-\gamma x^2} \sum_{l=0}^{\frac{m}{2}} (\sqrt{\gamma})^{2l} c_{m,l} x^{2l}, \quad \text{where}\,\,  c_{m,l}=  \frac{m! (2\sqrt{2})^{2l} (-1)^{\frac{m}{2}-l}}{(2l)! (\frac{m}{2}-l)!},                                            
\label{hermite}
\end{equation}
\begin{equation}
\text{for odd $m$:} \quad 
\phi_{m}(x) =
 e^{-\gamma x^2} \sum_{l=0}^{\frac{m-1}{2}} (\sqrt{\gamma})^{2l+1} c_{m,l} x^{2l+1}, \quad \text{where}\,\,  c_{m,l}=  \frac{m! (2\sqrt{2})^{2l+1} (-1)^{\frac{m-1}{2}-l}}{(2l+1)! (\frac{m-1}{2}-l)!},                                            
\label{hermiteodd}
\end{equation}
for $x \in \Z$ and $m=0,1,\dots$. 
The discrete, periodic Hermite functions of period $n$, denoted by $\hermmn$, are similar to $\hermm$, except that the range for
$x$ is $x=-\frac{n}{2},\dots,\frac{n}{2}-1$ (with periodic boundary conditions) and the range for $m$ is $m=0,\dots,n-1$.
\end{definition}

We denote the finite almost Mathieu operator, acting on sequences of length $n$, by $\Topnt$. It can be represented
by an $n \times n$ tridiagonal matrix, which has ones on the two side-diagonals and $\cos(x_{-\frac{n}{2}+k} + \theta)$ as  $k$-th entry
on its main diagonal for $x_k = 2\pi \alpha k$, $k=-\frac{n}{2},\dots,\frac{n}{2}-1$, and $j=0,\dots,n-1$. Here, we have assumed for simplicity  that $n$ is even, 
the required modification for odd $n$ is obvious.
Sometimes it is convenient to replace the translation present in the infinite almost Mathieu operator by a periodic translation. 
In this case we obtain  the $n\times n$ periodic almost Mathieu operator $\Hopnt$ which is  almost a tridiagonal matrix; it is given by
\begin{equation}
\Hopnt = 
\begin{bmatrix}
\cos(x_{-\frac{n}{2}}+\theta) & 1 & 0 & \dots & 0 & 1 \\
1 & \cos(x_{-\frac{n}{2}+1}+\theta) & 1 & 0 & \dots & 0  \\
0 & 1 & \cos(x_{-\frac{n}{2}+2}+\theta) & 1 &  \dots & 0  \\
\vdots & & & \ddots & & \vdots  \\
0 & \dots & 0 & 1 & \cos(x_{\frac{n}{2}-2}+\theta) & 1  \\
1 & 0 & \dots & 0 & 1 & \cos(x_{\frac{n}{2}-1}+\theta)
\end{bmatrix},
\label{finitemathieu}
\end{equation}
where $x_k = 2\pi \alpha k$ for  $k=-\frac{n}{2},\dots,\frac{n}{2}-1$. If $\theta =0$ we write $\Hopn$ instead of $\Hopnt$.

\section{Finite Hermite functions as approximate eigenvectors of the finite-dimensional almost Mathieu operator}
\label{s:finite}

In this section we focus on  eigenvector and eigenvalue estimates for the finite-dimensional model of the almost Mathieu operator.
Finite versions of $\Hopt$ are interesting in their own right. On the one hand, numerical simulations are frequently based on truncated versions of $\Hopt$, on the other hand certain problems, such as the study of  random walks on the Heisenberg group, are often more naturally carried out in the finite setting. 

We first gather some properties of the true eigenvectors of the almost Mathieu operator, collected in the following proposition.

\begin{proposition}
\label{prop:prop}
Consider the finite, non-periodic almost Mathieu operator $\Topnt$. Let $\zeta_0 \geq \zeta_1 \geq \cdots \geq \zeta_{n-1}$ be its eigenvalues and $\psi_0, \psi_1, \dots, \psi_{n-1}$ the associated eigenvectors. The following statements hold:
\begin{enumerate}
\item
$\Topnt \psi_m = \zeta_m\psi_m$ for all $0 \leq m \leq n-1$;
\item
There exist constants $C_1,C_2$, independent of $n$ such that, for all $m$
\[ |\psi_m(i)| \leq C_1e^{-C_2|i|}; \]
\item
If $\zeta_r$ is one of the $m$-th largest eigenvalues (allowing for multiplicity), then $\psi_r$ changes sign exactly $m$ times.
\end{enumerate}
\end{proposition}
\begin{proof}
The first statement is simply restating the definition of eigenvalues and eigenvectors. The second statement follows from the fact that the inverse of a tri-diagonal (or an almost tri-diagonal operator) exhibits exponential off-diagonal decay and the relationship between the spectral projection and eigenvectors associated to isolated eigenvalues. See \cite{BBR13} for details. The third statement is a result of Lemma \ref{lem:sturm} below.
\end{proof}

\begin{lemma}
\label{lem:sturm}
Let $A$ be a symmetric tridiagonal $n \times n$ matrix with entries $a_1,\dots,a_{n}$ on its main diagonal and
$b_1,\dots,b_{n-1}$ on its two non-zero side diagonals. Let $\lambda_0 \ge \lambda_2 \ge \dots \ge \lambda_m$  
be the eigenvalues of $A$ with multiplicities $r_0,\dots,r_m$. Let $v_0^{(1)},\dots,v_0^{(r_0)}, v_{1}^{(1)}, \dots, v_m^{(r_m)}$ be the associated eigenvectors. Then, for each $0 \leq i \leq m$ and each $1 \leq j_m \leq r_m$ the entries of the vector $v_{i}^{(j_m)}$ change signs $m$ times. That is, for all $1 \leq j_0 \leq r_0$, the entries of each of the vectors $v_{0}^{(j_0)}$ all have the same sign, while for all $1 \leq j_1 \leq r_1$ each of the vectors $v_{1}^{(j_1)}$ has only a single index where the sign of the entry at that index is different than the one before it, and so on.
\end{lemma}

\begin{proof}
This result follows directly from Theorem 6.1 in~\cite{ACE09} which relates the Sturm sequence to the sequence
of ratios of eigenvector elements. Using the assumption that $b_i > 0 $ for all $i=1,\dots,n$ yields the claim.
\end{proof}

The main results of this paper are summarized in the following theorem. In essence we show that the Hermite functions (approximately) satisfy all three
eigenvector properties listed in Proposition~\ref{prop:prop}. The technical details are presented later in this section.
\begin{theorem}
\label{th:newmain}
Let $A$ be either of the operators $\Hopnt$ or $\Topnt$. Let $\hermmn$ be as defined in Definition~\ref{hermdef}. Set $\gamma = \pi\alpha\sqrt{\beta}$, let $0 < \varepsilon < 1$, and assume $\frac{4}{n^{2-\varepsilon}} < \gamma < \frac{1}{n^\varepsilon}$. Then, for $m = 0,\dots,N$ where $N = \mathcal{O}(1)$, the following statements hold:
\begin{enumerate}
\item
$A\hermmn \approx \lambda_m\hermmn$, where $\lambda_m = 2\beta + 2 e^{-\gamma} - 4m \gamma e^{-\gamma}$;
\item
For each $m$, there exist constants $C_1,C_2$, independent of $n$ such that
\[ |\hermmn(i)| \leq C_1e^{-C_2|i|}; \]
\item
For each $0 \leq m \leq n-1$, the entries of $\hermmn$ change signs exactly $m$ times.
\end{enumerate}
\end{theorem}
\begin{proof}
The second property is an obvious consequence of the definition of $\hermmn$, while the first and third are proved in Theorem \ref{th:main} and Lemma \ref{lem:hermsigns}, respectively. That the theorem applies to both $\Hopnt$ and $\Topnt$ is a consequence of Corollary \ref{cor:both}.
\end{proof}
In particular,  the (truncated) Gaussian function is an approximate eigenvector associated with the
largest eigenvalue of $\Hopnt$ (this was also proven by Persi Diaconis~\cite{Persi}).

In fact, via the following symmetry property, Theorem \ref{th:newmain} also applies to the $m$ smallest eigenvalues of $\Hopnt$ and their associated eigenvectors.
\begin{proposition}
\label{le:negativeeigenvalues}
If $\varphi$ is an eigenvector of $\Hopnt$ with eigenvalue $\lambda$, then $M_{1/2}T_{1/2\alpha} \varphi$ is an eigenvector of $\Hopnt$
with eigenvalue $-\lambda$. 
\end{proposition}

\begin{proof}
We assume that $\theta = 0$, the proof for $\theta \neq 0$ is left to the reader.
It is convenient to express $\Hopn$ as
$$\Hopn = T_{1} + T_{-1} + \beta M_{\alpha} + \beta M_{-\alpha}.$$
Next we study the commutation relations between $\Hopn$ and translation and modulation by considering
\begin{equation}
T_a M_b ( T_{1} + T_{-1} + \beta M_{\alpha} + \beta M_{-\alpha}) (T_a M_b)^{\ast}.
\label{comm2}
\end{equation}
Using~\eqref{comm} we have
$$T_a M_b  T_{1}  M_{-m} T_{-a} = e^{2\pi i b} T_{1},     \qquad T_a M_b  T_{-1}  M_{-m} T_{-a} = e^{-2\pi i b} T_{-1}, $$
and 
$$T_a M_b  M_{\alpha}  M_{-b} T_{-a} = e^{-2\pi i \alpha  a} M_{\alpha}, \qquad T_a M_b  M_{-\alpha}  M_{-m} T_{-a} = e^{2\pi i \alpha  a} M_{-\alpha}.$$
Note that $e^{2\pi i b} =  \pm 1$ if $b \in \frac{1}{2} \Z$ and $e^{-2\pi i \alpha  a} =  \pm 1$ if  $a \alpha \in \frac{1}{2} \Z$. In particular, if $b=\frac{1}{2}$ and $a = \frac{1}{2\alpha}$ we obtain
$$T_\frac{1}{2\alpha} M_\frac{1}{2} \Hop (T_\frac{1}{2\alpha} M_\frac{1}{2})^{\ast} = - \Hop.$$
Hence, since $\Hopn T_\frac{1}{2\alpha} M_\frac{1}{2} v = - T_\frac{1}{2\alpha} M_\frac{1}{2} \Hopn v$ for any $v\in \C^n$, 
it follows that if $\varphi$ is an eigenvector of $\Hopn$ with eigenvalue $\lambda$, then $T_\frac{1}{2\alpha} M_\frac{1}{2}\varphi $ is an eigenvector 
of $\Hopn$ with eigenvalue $-\lambda$. Finally, note that if the entries of $\varphi$ change signs $k$ times, then the entries of $M_{\frac{1}{2}} \varphi$ change sign exactly $n-1 - k$ times, while the translation operator does not affect the signs of the entries.
\end{proof}
The attentive reader will note that the proposition above also holds for the infinite-dimensional almost Mathieu operator $\Hopt$.

In the following lemma we establish an identity about the coefficients $c_{m,l}$ in~\eqref{hermite} and~\eqref{hermiteodd}, and
the binomial coefficients $b_{2l,k}$, which we will need later in the proof of Theorem~\ref{th:main}.

\begin{lemma}
\label{le:coeff}
For $m=0,\dots,n-1; l = 0,\dots,\frac{m}{2}$, there holds
\begin{equation}
\label{coeff1}
2c_{m,l} b_{2 l,2} - 4 c_{m,l-1} b_{2l-2,1} = -4m c_{m,l-1},
\end{equation}
where $b_{j,k} = {j \choose k} = \frac{j!}{k!(j-k)!}$.
\end{lemma}
\begin{proof}
To verify the claim we first note that $b_{2l,2} = \frac{2l(2l-1)}{2}$ and $b_{2l-2,1} = 2l-2$. Next, note that ~\eqref{coeff1} is equivalent to 
\begin{equation}
\label{coeff2}
c_{m,l} = 4\frac{2l - 2 - m}{2l(2l-1)} c_{m,l-1}.
\end{equation}
Now, for even $m$ we calculate
\begin{align*}
c_{m,l} & = m!  \frac{(2\sqrt{2})^{2l} (-1)^{\frac{m}{2}-l}}{(2l)! (\frac{m}{2}-l)!} = \frac{-8\left(\frac{m}{2} - l +1  \right) m! (2\sqrt{2})^{2l-2} (-1)^{\frac{m}{2} - l +1}}{2l(2l-1)(2l-2)! \left( \frac{m}{2} -l +1 \right)} \\
& =  \frac{-8\left(\frac{m}{2} - l +1  \right) }{2l(2l-1)} c_{m,l-1} = 4\frac{2l - 2 - m}{2l(2l-1)} c_{m,l-1},
\end{align*}
as desired. The calculation is almost identical for odd $m$, and is left to the reader.                             
\end{proof}



While the theorem below is stated for general $\alpha, \beta,\theta$ (with some mild conditions on $\alpha,\beta$), it is most instructive
to first consider the statement for $\alpha=\frac{1}{n},\beta=1,\theta=0$. In this case the parameter $\gamma$ appearing below will
take the value $\gamma = \frac{\pi}{n}$. The theorem then states that at the right edge of the spectrum of $\Hopnt$, $\hermmn$ is an approximate eigenvector of $\Hopnt$ with approximate eigenvalue $\lambda_m = 2\beta + 2 e^{-\gamma} - 4m \gamma e^{-\gamma}$, 
and a similar result holds for the left edge. The error is of the order $\frac{1}{n^2}$ and the edge of the spectrum is of size $\bigO(1)$. 
If we allow the approximation error to increase to be of order $\frac{1}{n}$, then the size of the edge of the spectrum will increase 
to  $\bigO(\sqrt{n})$. 

\begin{theorem}
\label{th:main}
Let $\Hopnt$ be defined as in~\eqref{finitemathieu} and let $\alpha, \beta \in \R+$ and $\theta \in [0,2\pi)$. Set $\gamma = \pi \alpha
\sqrt{\beta}$  and assume that $\frac{4}{n^2} \le \gamma < 1$. \\
(1) For $m = 0,1,\dots, N$, where $N=\bigO(1)$, there holds for all $x =-\frac{n}{2},\dots,\frac{n}{2}-1$
\begin{equation}
\label{mainbound}
\left|(\Hopnt \hermmn)(x) - \lambda_m  \hermmn\left(x + \frac{\theta}{2\pi\alpha}\right)\right| \le \bigO(\gamma^2),
\end{equation}
where 
\begin{equation}
\label{eigenbound}
\lambda_m = 2\beta + 2 e^{-\gamma} - 4m \gamma e^{-\gamma}.
\end{equation}
(2) For $m =-n+1,-n+2,\dots,-N$, where $N=\bigO(1)$, there holds for all 
$x =-\frac{n}{2},\dots,\frac{n}{2}-1$
\begin{equation}
\label{mainbound2}
\left|(\Hopnt \hermmn)(x) - \lambda_m  (-1)^{x} \hermmn\left(x + \frac{\theta}{2\pi\alpha} - \frac{1}{2\alpha}\right)\right| \le \bigO(\gamma^2),
\end{equation}
where
\begin{equation}
\label{eigenbound2}
\lambda_m =  - (2\beta + 2 e^{-\gamma} - 4m \gamma e^{-\gamma}).
\end{equation}
\end{theorem}

\begin{proof}
We prove the result for $\theta =0$ and for even $m$; the proofs for  $\theta \neq 0$ and for odd $m$ are similar and left to the reader. 
Furthermore, for simplicity of notation, throughout the proof
we will write $H$ instead of $\Hopn$ and $\hermm$ instead of $\hermmn$.

We compute for $x =-\frac{n}{2},\dots,\frac{n}{2}-1$ (recall that $\hermm(\frac{n}{2}-1) = \herm(-\frac{n}{2})$ due to our assumption of
periodic boundary conditions)
\begin{align}
 (H \hermm)(x)  & =   \hermm(x+1) + \hermm(x-1) + 2\beta \cos (2\pi \alpha x)\big) \hermm(x) &\\
& =   e^{-\gamma x^2} e^{-\gamma} \Big(e^{-2\gamma x} \cdot \sum_l c_{m,l} \gamma^l (x+1)^{2l} + 
   e^{2\gamma x} \cdot \sum_l c_{m,l} \gamma^l (x-1)^{2l} \Big)+& \\
  &\quad  + e^{-\gamma x^2}  2\beta \cos(2\pi \alpha x) \sum_l c_{m,l} \gamma^l x^{2l}.&
\end{align}
We expand each of the terms $(x\pm 1)^{2l}$ into its binomial series, i.e.,
\begin{equation}
(x+1)^{2l} = \sum_{k=0}^{2l} b_{2l,k} x^{2l-k} 1^k \quad \text{and} \,\, (x-1)^{2l} = \sum_{k=0}^{2l} b_{2l,k} x^{2l-k} (-1)^k, \quad
\text{where}\,\, b_{2l,k}  = \left(\begin{matrix} 2l \\ k \end{matrix}\right),
\end{equation}
and obtain after some simple calculations
\begin{align}
 (H \hermm)(x)   = \,\,&    e^{-\gamma x^2} \Big[ \sum_{l=0}^{m/2} c_{m,l} \gamma^l  x^{2l} \Big( e^{-\gamma} b_{2l,0}(e^{-2\gamma x}+e^{2\gamma x})
 + 2\beta \cos(2\pi \alpha x) \Big) \Big] +& \label{termI} \\
& +  e^{-\gamma x^2} \Big[ \sum_{l=0}^{m/2} c_{m,l} \gamma^l \sum_{k=1}^{2l} b_{2l,k} x^{2l-k}  e^{-\gamma} (e^{-2\gamma x}+ (-1)^{k} e^{2\gamma x})\Big]& \\ 
& = \text{(I)}\, + \text{(II)} .&  \label{termII}
 \end{align}
We will now show that $\text{(I)} = \hermm(x) \cdot (2\beta + 2e^{-\gamma}) + \bigO(\gamma^2)$ and
 $\text{(II)} = -\hermm(x) \cdot 4 m \gamma e^{-\gamma}+ \bigO(\gamma^2)$, from which \eqref{mainbound} and \eqref{eigenbound} will follow.

We first consider the term (I). We rewrite (I) as
$$\text{(I)}  \, = e^{-\gamma x^2}  \Big( e^{-\gamma} (e^{-2\gamma x}+e^{2\gamma x})    +   2\beta \cos(2\pi \alpha x) \Big)  \sum_{l=0}^{m/2} c_{m,l} \gamma^l  x^{2l}.$$
Using  Taylor approximations for  $e^{-2\gamma x}, e^{2\gamma x}$, and
 $\cos(2\pi \alpha x)$ respectively, we obtain after some rearrangements 
 (which are justified due to the absolute summability of each of the involved infinite series)
 \begin{gather}
 e^{-\gamma x^2}  \Big( e^{-\gamma} (e^{-2\gamma x}+e^{2\gamma x}) + 2\beta \cos(2\pi \alpha x) \Big)  = \notag \\
=e^{-\gamma x^2}   e^{-\gamma} \Big[\big(1+ (-2\gamma x)+ \frac{(-2\gamma x)^2}{2!} + \frac{(-2\gamma x)^3}{3!} + R_1(x) \big) \\ \notag +
\big(1+ (2\gamma x)+ \frac{(2\gamma x)^2}{2!} + \frac{(2\gamma x)^3}{3!}  + R_2(x) \big)  \Big]  \\
   +  e^{-\gamma x^2} 2\beta \big(1+\frac{(2\pi \alpha x)^2}{2!} + R_3(x)  \big)   \notag \\
 = e^{-\gamma x^2} \Big(2e^{-\gamma} + 2e^{-\gamma} \frac{(2\gamma x)^2}{2!} + 2\beta - 2\beta \frac{(2\pi \alpha x)^2}{2!}
 +   e^{-\gamma} \big(R_1(x) + R_2(x)\big)   + 2\beta  R_3(x) \Big), \label{sum1}
 \end{gather}
where $R_1, R_2$, and $R_3$ are the remainder terms of the Taylor expansion for
$ e^{-2\gamma x}, e^{2\gamma x}$, and $\cos(2\pi \alpha x)$ (in this order) respectively, given by

$$R_1(x) =  \frac{(-2\gamma)^4 e^{-2\gamma \xi_1}}{4!}x^4, \quad
R_2(x) =  \frac{(2\gamma)^4 e^{2\gamma \xi_2}}{4!} x^4, \quad R_3(x) = \frac{(2\pi \alpha)^4 \cos(\xi_3)}{4!}x^4,$$
with real numbers $\xi_1, \xi_2, \xi_3$  between 0 and $x$.
We use a second-order Taylor approximation for $e^{-\gamma}$ with corresponding remainder term 
$R_4(\gamma) = \frac{e^{-\xi_4}}{3!}\gamma^3$ (for some $\xi_4 \in (0,\gamma)$) in~\eqref{sum1}.

 Hence~\eqref{sum1} becomes 
 \begin{gather}
 e^{-\gamma x^2} \left(2e^{-\gamma} + (2\gamma x)^2 +\left(-\gamma + \frac{\gamma^2}{2!} + R_4(\gamma)\right)(2\gamma x)^2 + 2\beta - \beta (2\pi \alpha x)^2 \right) +\label{sum2} \\
 +    e^{-\gamma x^2} \left( \left(1-\gamma + \frac{\gamma^2}{2!} + R_4(\gamma)\right) \big(R_1(x) + R_2(x)\big)   + 2\beta  R_3(x) \right).
\label{sum3}     
\end{gather}
Since $\gamma = \pi \alpha \sqrt{\beta}$, the terms  $ (2\gamma x)^2 $ and $-\beta (2 \pi \alpha x)^2$ in~\eqref{sum2}
cancel.
Clearly, $|R_1(x)|  \le \frac{(2\gamma x)^4}{4!}$,
$|R_2(x)|  \le \frac{e^{2  \gamma x}(2\gamma x)^4}{4!}$, and $|R_4(\gamma)|  \le \frac{(\gamma)^3}{3!}$.
It is convenient to substitute $\alpha = \frac{\gamma}{\pi\sqrt{\beta}}$
in $R_3(x)$, in which case we get $|R_3(x)| \le \frac{(2\gamma x)^4 }{\beta ^2 4!}$.
Thus, we can bound the expression in~\eqref{sum3} from above by
 \begin{align}
 \Big|  e^{-\gamma x^2} \Big( & \big(-\gamma + \frac{\gamma^2}{2!} + R_4(\gamma)\big)  \big(R_1(x) + R_2(x)\big)   + 2\beta  R_3(x) \Big) \Big|  \\
 &\le e^{-\gamma x^2} \Big( \big(\gamma + \frac{\gamma^2}{2!} + \frac{(\gamma)^3}{3!}\big)  \big(\frac{(2\gamma x)^4}{4!} + 
 \frac{e^{2  \gamma x}(2\gamma x)^4}{4!}   + 2\beta  \frac{(2\gamma x)^4 }{\beta ^2 4!}  \big) \Big).
\label{sum4}     
\end{align}
Assume now that $|x| \le \frac{1}{\sqrt{\gamma}}$ then we can further bound  the expression in \eqref{sum4} from above by
 \begin{align}
 e^{-\gamma x^2} \Big( & \big(\gamma + \frac{\gamma^2}{2!} + \frac{\gamma^3}{3!}\big)  \big(\frac{(2\gamma x)^4}{4!} + 
 \frac{e^{2  \gamma x}(2\gamma x)^4}{4!}   + 2\beta  \frac{(2\gamma x)^4 }{\beta ^2 4!}  \big) \Big)  \\
& \le  \Big( \big(\gamma + \frac{\gamma^2}{2!} + \frac{\gamma^3}{3!}\big)  \big(\frac{2 \gamma^2}{3} + 
 \frac{ 2 e^{2  \sqrt{\gamma}} \gamma^2}{3}   +   \frac{4 \gamma^2 }{3\beta }  \big) \Big)\le \bigO(\gamma^3) + \bigO(\gamma^3/\beta).
 \end{align}
Moreover, if  $|x| \le \frac{1}{\sqrt{\gamma}}$,  we can bound  the term 
$\big(-\gamma + \frac{\gamma^2}{2!} + R_4(\gamma)\big)\frac{(2\gamma x)^2}{2!}$ in~\eqref{sum2} by
$$\Big| \big(-\gamma + \frac{\gamma^2}{2!} + R_4(\gamma)\big)\frac{(2\gamma x)^2}{2!}\Big| \le 
2\gamma^2+ \gamma^3 + \frac{\gamma^4}{3} \le \bigO(\gamma^2).$$

Now suppose $|x| \ge \sqrt{\frac{1}{\gamma}}$. We set $|x|  = \left(\sqrt{\frac{1}{\gamma}} \,\right)^c$ for some $c$ with
\begin{equation}
1 < c \le \frac{2\log (n/2)}{\log (1/\gamma)}.
\label{condc}
\end{equation} 
The upper bound in~\eqref{condc} ensures that  $|x| \le \frac{n}{2}$ and the assumption
$\frac{4}{n^2} \le \gamma < 1$ implies that condition~\eqref{condc} is not empty. Then
\begin{gather}
e^{-\gamma x^2} \Big( \big(\gamma + \frac{\gamma^2}{2!} + \frac{(\gamma)^3}{3!}\big)  \big(\frac{(2\gamma x)^4}{4!} + 
 \frac{e^{2  \gamma x}(2\gamma x)^4}{4!}   + 2\beta  \frac{(2\gamma x)^4 }{\beta ^2 4!}  \big) \Big) = \label{bound8} \\
 =  e^{-\gamma^{1-c}} \Big( \big(\gamma + \frac{\gamma^2}{2!} + \frac{(\gamma)^3}{3!}\big)  \big(\frac{(2\gamma)^{4-2c}}{4!} + 
 \frac{e^{2  \gamma^{1-c}}(2\gamma)^{4-2c}}{4!}   +   \frac{4 \gamma^{4-2c} }{3\beta }  \big) \Big) \le \bigO(\gamma^2),
 \label{bound9}
 \end{gather}
 where the last inequality follows from basic inequalities like  $e^{-\gamma^{1-c}} \gamma^{4-2c} \le \gamma^{2}$ (which in turn
 follows from $e^{y} \ge y^2$ for all $y\ge0$).

Thus we have shown that
 \begin{equation}
 e^{-\gamma x^2}  \Big( e^{-\gamma} (e^{-2\gamma x}+e^{2\gamma x}) + 2\beta \cos(2\pi \alpha x) \Big)  = 
 e^{-\gamma x^2}  \big( 2e^{-\gamma} + 2\beta \big) + \bigO(\gamma^2).
 \label{estimate1}
 \end{equation}
 Returning to the term (I) in \eqref{termI}, we obtain, using~\eqref{estimate1}, 
 \begin{gather}
 \label{bound10}
 \text{(I)}\,\, = \hermm(x)  (2e^{-\gamma} + 2\beta) +  C_{\gamma} \sum_{l=0}^{m/2} c_{m,l} \gamma^l x^{2l}
  \end{gather}
where $C_{\gamma} =    \bigO(\gamma^2)$.

Let us analyze the error term $C_{\gamma} \sum_{l=0}^{m/2} c_{m,l} \gamma^l  x^{2l}$. Using Stirling's Formula (\cite[Page 257]{ASbook})
we note that $c_{m,l}$ grows at least as fast as $\big(m/e\big)^{m/2}$, but not faster than $(m/e)^{m}$. Hence, 
the error term will remain of size $\bigO(\gamma^2)$, as long as we ensure that $m$ does not exceed $\bigO(1)$.

\medskip
We now proceed to showing that $\text{(II)} = -\hermm(x) \cdot \big(4 m \gamma e^{-\gamma}\big) + \bigO(\gamma^2)$.
The key to this part of the proof is the observation that $H$ acts ``locally'' on the powers $x^{2l}$ that appear in the definition of $\hermm$.
Recall that the term (II) has the form
\begin{equation}
e^{-\gamma x^2} \Big[ \sum_{l=0}^{m/2} c_{m,l} \gamma^l \sum_{k=1}^{2l} b_{2l,k} x^{2l-k}  e^{-\gamma} (e^{-2\gamma x}+ (-1)^{k} e^{2\gamma x})\Big].
\label{term2again}
\end{equation}
Analogous to the calculations leading up to~\eqref{estimate1} we can show that for odd $k$ there holds
\begin{equation}
e^{-\gamma x^2} e^{-\gamma} (e^{-2\gamma x}+ (-1)^{k} e^{2\gamma x}) = (-4\gamma x)e^{-\gamma x^2} e^{-\gamma} + \bigO(\gamma^2),
\label{oddk}
\end{equation}
and for even $k$
\begin{equation}
e^{-\gamma x^2} e^{-\gamma} (e^{-2\gamma x}+ (-1)^{k} e^{2\gamma x}) = 2e^{-\gamma x^2} e^{-\gamma} + \bigO(\gamma^2).
\label{evenk}
\end{equation}
Using~\eqref{oddk} and \eqref{evenk} we can express~\eqref{term2again} as 
\begin{equation}
\big(e^{-\gamma x^2} e^{-\gamma}  + \bigO(\gamma^2)\big) \Big[ \sum_{l=0}^{m/2} c_{m,l} \gamma^l \big(\sum_{\text{odd $k$}} (-4\gamma x)  b_{2l,k} x^{2l-k}   
+ \sum_{\text{even $k$}} 2 b_{2l,k} x^{2l-k}  \big) \Big].
\label{estimate2}
\end{equation}
Furthermore, using estimates similar to the ones used in deriving the bounds for (I), one easily verifies that
$$e^{-\gamma x^2} e^{-\gamma}  \gamma^l  (-4\gamma x)  x^{2l-k}   \le \bigO(\gamma^2),\qquad
\text{for $k \ge 3$},$$ 
and 
$$e^{-\gamma x^2} e^{-\gamma} c_{m,l} \gamma^l  x^{2l-2}  \le  \bigO(\gamma^2), \qquad
\text{for $k \ge 4$}.$$
Therefore,
\begin{gather}
\label{estimate3}
e^{-\gamma x^2} e^{-\gamma} \Big[ \sum_{l=0}^{m/2} c_{m,l} \gamma^l \big(\sum_{\text{odd $k$}} (-4\gamma x)  b_{2l,k} x^{2l-k}   
+ \sum_{\text{even $k$}} 2 b_{2l,k} x^{2l-k}  \big) \Big]  = \\
\big( e^{-\gamma x^2} e^{-\gamma} + \bigO(\gamma^2) \big) \Big[ \sum_{l=0}^{m/2} c_{m,l} \gamma^l \big( (-4\gamma x)  b_{2l,1} x^{2l-1}   
+  2 b_{2l,2} x^{2l-2}  \big) \Big] .
\label{estimate4}
\end{gather}
The expression above implies that $H$ acts ``locally'' on the powers $x^{2l}$ of $\hermm$.

Moreover,
\begin{gather}
 e^{-\gamma x^2} e^{-\gamma} \Big[ \sum_{l=0}^{m/2} c_{m,l} \gamma^l \big( (-4\gamma x)  b_{2l,1} x^{2l-1}   
+  2 b_{2l,2} x^{2l-2}  \big) \Big]  \\
=e^{-\gamma x^2} e^{-\gamma} \Big[ (-4\gamma m)   c_{m,\frac{m}{2}} \gamma^{m/2}   x^{m}    
 + 2 c_{m,\frac{m}{2}} \gamma^{m/2}  b_{m,2} x^{m-2}   \\
 + \sum_{l=0}^{m/2-1} (-4\gamma) c_{m,l} \gamma^l   b_{2l,1} x^{2l}   
+  2 c_{m,l} \gamma^l b_{2l,2} x^{2l-2}  \big) \Big],
\label{estimate5}
\end{gather}
and
\begin{gather}
(-4m \gamma)   c_{m,\frac{m}{2}} \gamma^{m/2}   x^{m}  +
 2 c_{m,\frac{m}{2}} \gamma^{m/2}  b_{m,2} x^{m-2}  + 
 \sum_{l=0}^{m/2-1} (-4\gamma) c_{m,l} \gamma^l   b_{2l,1} x^{2l}   
+  2 c_{m,l} \gamma^l b_{2l,2} x^{2l-2}  \big)   \notag \\
= (-4m \gamma)   c_{m,\frac{m}{2}} \gamma^{m/2}   x^{m}  +
\sum_{l=0}^{m/2} \big( 2 c_{m,l} \gamma^{l}  b_{2l,2} x^{2l-2}  
 -4\gamma c_{m,l-1} \gamma^{l-1}   b_{2l-2,1} x^{2l-2} \big) \\
 =    (-4m \gamma )   c_{m,\frac{m}{2}} \gamma^{m/2}   x^{m}  +
\gamma \sum_{l=0}^{m/2} \gamma^{l-1} x^{2l-2} \big(2 c_{m,l}   b_{2l,2}    -4 c_{m,l-1}    b_{2l-2,1} \big)  \\
= (-4m \gamma )   c_{m,\frac{m}{2}} \gamma^{m/2}   x^{m}  -
4m \gamma \sum_{l=0}^{m/2} c_{m,l-1} \gamma^{l-1} x^{2l-2}  \label{estimate6} \\
= 4m \gamma \sum_{l=0}^{m/2} c_{m,l} \gamma^{l} x^{2l},
\label{estimate7}
\end{gather}
where we have used Lemma~\ref{le:coeff} in \eqref{estimate6}. Hence, up to an error of size $\bigO(\gamma^2)$ we have
\begin{equation}
e^{-\gamma x^2} e^{-\gamma} \Big[ \sum_{l=0}^{m/2} c_{m,l} \gamma^l \big(\sum_{\text{odd $k$}} (-4\gamma x)  b_{2l,k} x^{2l-k}   
+ \sum_{\text{even $k$}} 2 b_{2l,k} x^{2l-k}  \big) \Big] 
= -4 m \gamma  e^{-\gamma} e^{-\gamma x^2}   \sum_{l=0}^{m/2} c_{m,l} \gamma^{l} x^{2l} .
\label{bound12}
\end{equation}

Analogous to the estimate of (I), we need to choose $m$ to be not larger than $\bigO(1)$, in order to keep the error term
\begin{equation}
\label{bound14}
\bigO(\gamma^2)  \Big[ \sum_{l=0}^{m/2} c_{m,l} \gamma^l \big( (-4\gamma x)  b_{2l,1} x^{2l-1}   
+  2 b_{2l,2} x^{2l-2}  \big) \Big]  
\end{equation}
of size $\bigO(\gamma^2)$.

Therefore, for such an $m$, by invoking~\eqref{hermite}, \eqref{estimate2}, and~\eqref{bound12} we can
express the  term (II) in \eqref{term2again} as
$$\text{(II)} \,\, = - 4 m \gamma  e^{-\gamma} \hermm(x) + \bigO(\gamma^2).$$
Hence we have shown that
$$(H \hermm)(x) = (2\beta + 2 e^{-\gamma} - 4m \gamma e^{-\gamma}) \hermm(x) + \bigO(\gamma^2),$$
which establishes claims~\eqref{mainbound} and \eqref{eigenbound}.

Claims~\eqref{mainbound2} and \eqref{eigenbound2}  of Theorem~\ref{th:main} follow now from Proposition~\ref{le:negativeeigenvalues}.
\end{proof}

\begin{remark}
If we allow the error bound in~\eqref{mainbound} and \eqref{mainbound2} to increase from $\bigO(\gamma^2)$ to  $\bigO(\gamma)$, then the range
for which the $\hermm$ serve as approximate eigenfuctions for $\Hopnt$ extends from $m \sim \bigO(1)$ to $m \sim \bigO(\sqrt{1/\gamma})$. 
This can be seen by a suitable modification of the proof of Theorem~\ref{th:main} in the derivations in~\eqref{bound8},
\eqref{bound9} and in the error term~\eqref{bound10}, as well as in the calculations leading up to \eqref{bound12} and in the error term~\eqref{bound14}.
(We leave the details to the reader.) See also the numerical simulations depicted in Figure~\ref{fig3}, which illustrate this fact.
\end{remark}

\begin{corollary}
\label{cor:both}
Theorem \ref{th:main} also holds if we replace $\Hopnt$ with $\Topnt$.
\end{corollary}

\begin{proof}
A simple application of the triangle inequality yields
\[ \left\| \Topnt \hermmn - \lambda_m\hermmn \right\|_2 \leq \left\| \Topnt \hermmn - \Hopnt\hermmn \right\|_2 + \left\| \Hopnt\hermmn - \lambda_m\hermmn \right\|_2. \]
We compute
\begin{gather}
\left\| \Topnt \hermmn -  \Hopnt\hermmn \right\|_2^2 = \left| \hermmn\left( \frac{n}{2} -1\right) \right|^2 + \left| \hermmn\left( -\frac{n}{2}\right) \right|^2 \\
 = \left| e^{-\gamma\left( \frac{n}{2} -1\right)^2} \sum_{l = 0}^{m/2} \gamma^l c_{m,l}\left( \frac{n}{2} -1\right)^{2l} \right|^2 + \left| e^{-\gamma\left( -\frac{n}{2}\right)^2} \sum_{l = 0}^{m/2} \gamma^l c_{m,l}\left( - \frac{n}{2} \right)^{2l} \right|^2. 
\end{gather}
Now, from \ref{bound10} we know that the sums are of size $\mathcal{O}(\gamma^2)$, so we conclude (in a large overestimate when $n$ is large) that
\[ \left\| \Topnt \hermmn -  \Hopnt\hermmn \right\|_2 = \mathcal{O}(\gamma^2).  \]
Meanwhile, from Theorem \ref{th:main}, we know that 
\[ \left\| \Hopnt\hermmn - \lambda_m\hermmn \right\|_2 = \mathcal{O}(\gamma^2). \]
\end{proof}


\begin{lemma}
\label{lem:hermsigns}
Let $\hermmn$ be as before. Let $0 < \varepsilon < 1$ and assume $\frac{4}{n^{2-\varepsilon}} \leq \gamma < \frac{1}{n^\varepsilon}$. Then, as long as $m < n^\varepsilon -1$, the entries of $\hermmn$ change signs $m$ times. 
\end{lemma}
\begin{proof}
Recall that the Hermite polynomial of order $m$ has $m$ distinct real roots. From \cite{Krasikov01}, we know that the roots of the $m$th Hermite polynomial lie in the interval $[-\sqrt{2m +2},\sqrt{2m+2}]$. Now, in our definition of $\hermmn$, we have applied the transformation $x \mapsto \sqrt{2\gamma} x$ to the Hermite polynomials. Thus, in our case, we see that the zeroes of $\hermmn$ lie in the interval 
\[ \left[ -\sqrt{\frac{m+1}{\gamma}},\sqrt{\frac{m+1}{\gamma}} \right]. \]
Now, the condition $\frac{4}{n^{2-\varepsilon}} \leq \gamma$ implies that the zeroes of $\hermmn$ lie in the interval
\[ \left[-\frac{n^{1-\frac{\varepsilon}{2}}}{2}\sqrt{m+1},\frac{n^{1-\frac{\varepsilon}{2}}}{2}\sqrt{m+1}  \right], \]
from which we see that so long as $m \leq n^{\varepsilon}-1$, all of the zeroes of the $m$th Hermite polynomial lie in the interval
\[ \left[ -\frac{n}{2}, \frac{n}{2} \right].  \]
Then, if the distance between consecutive zeroes is always larger than 1, the lemma will follow. Now, from \cite{szego67} we know that the distance between consecutive roots of the $m$th Hermite polynomial is at least $\frac{\pi}{\sqrt{2m+1}}$. Applying our scaling, we see that the distance between consecutive zeroes of $\hermmn$ is at least $\frac{\pi}{\sqrt{2\gamma(2m+1)}}$. The condition $\gamma < \frac{1}{n^\varepsilon} $ leads to 
\[ \frac{\pi}{\sqrt{2\gamma(2m+1)}} > \pi\sqrt{\frac{n^\varepsilon}{4m+2}}, \]
and so we see that whenever $m < \frac{\pi^2n^\varepsilon-2}{4}$, the minimum distance between consecutive zeroes of $\hermmn$ will be greater than 1. Finally, note that $n^\varepsilon -1 < \frac{\pi^2n^\varepsilon-2}{4}$.

\end{proof}
In particular, by choosing $\varepsilon = \frac{1}{2}$ in Lemma \ref{lem:hermsigns}, and thus $\frac{4}{\sqrt{n^3}} \leq \gamma \leq \frac{1}{\sqrt{n}}$, we obtain that for $m < \sqrt{n}+1$, $\hermmn$ changes sign $m$ times.

\begin{remark} Let  $\alpha = \frac{r}{n}$, where $r \in \N_{+}$.   Then it is easy to see that each eigenvalue of $\Hopn$ has multiplicity (at least) $\frac{n}{r}$ and
if $f$ is an approximate eigenvector of $\Hopn$ with eigenvalue $\lambda$, then so is $T_{k n/r} f$, for $k=1,\dots,r-1$.  
\end{remark}

\section{Eigenanalysis of the infinite almost Mathieu operator $\Hopt$}
\label{s:infinite}

In the preceding sections we have seen that $\hermmn$ are approximate eigenfunctions with associated eigenvalues $\lambda_m$ for both $\Hopnt$ and $\Topnt$. We now consider the infinite dimensional operator $\Hopt$ and its relationship with the non-periodic version of the Hermite functions defined in \eqref{hermite} and \eqref{hermiteodd}. We shall denote these functions by $\varphi_m$, and for all $m \geq 0$ they are defined exactly as in Definition \ref{hermdef} except that the range of $x$ is $\Z$. Note that the claims in Proposition \ref{prop:prop} no longer make sense for the infinite dimensional operator $\Hopt$. Nevertheless, we see that $\lambda_m$ and $\varphi_m$ still function as approximate eigenpairs for $\Hopt$, and the statements of Theorem \ref{th:newmain} still hold.

\begin{theorem}
\label{th:maininf}
Let $\Hopt, \varphi_m, \lambda_m$ be as before. Set $\gamma = \pi\alpha\sqrt{\beta}$ and assume $0 < \gamma < 1$. Then, for $m = 0,\dots,N$, where $N = \mathcal{O}(1)$, the following statements hold:
\begin{enumerate}
\item
$\Hopt\varphi_m \approx \lambda_m\varphi_m;$
\item
For each $m$, there exist constants $C_1,C_2 > 0$ such that
\[ |\varphi_m(i)| \leq C_1e^{-C_2|i|}; \]
\item
For all $0 \leq m < \frac{\pi^2}{4\gamma} - \frac{1}{2}$, the entries of $\varphi_m$ change signs exactly $m$ times.
\end{enumerate}
\end{theorem}
\begin{proof}
The first statement is proved in Corollary \ref{cor:inf} to Theorem \ref{th:main} below. The second follows from the definition of the Hermite functions. The third statement is an easy consequence of Lemma \ref{lem:hermsigns}. 
\end{proof}

\begin{corollary}
\label{cor:inf}
Let $\Hopt$ be as defined in \eqref{mathieu} and let $\alpha,\beta \in \mathbb{R}^+$ and $\theta \in [0,2\pi)$. Set $\gamma = \pi\alpha\sqrt{\beta}$ and assume that $0 < \gamma < 1$. \\
(1) For $m = 0,1,\dots, N$, where $N=\bigO(1)$, $T_{\frac{-\theta}{2\pi\alpha}} \phi_m$ is an approximate eigenvector for $\Hopt$ with approximate eigenvalue 
$\lambda_m = 2\beta + 2 e^{-\gamma} - 4m \gamma e^{-\gamma}$. Moreover, there holds for all $x \in \mathbb{Z}$
\begin{equation}
\label{infbound}
\left|(\Hopt \hermm)(x) - \lambda_m  T_{\frac{-\theta}{2\pi\alpha}}  \phi_m (x) \right| \le \bigO(\gamma^2).
\end{equation}
(2) For $m =-n+1,-n+2,\dots,-N$, where $N=\bigO(1)$, $T_{\frac{-\theta}{2\pi\alpha}}  M_{1/2} \phi_m$ is an approximate eigenvector and 
$\lambda_m = -(2\beta + 2 e^{-\gamma} - 4m \gamma e^{-\gamma})$ is an approximate eigenvalue associated with the $m$-th largest negative eigenvalue of $\Hop$.  Moreover, there holds for all $x \in \mathbb{Z}$
\begin{equation}
\label{infmainbound2}
\left|(\Hopt \hermm)(x) - \lambda_m  T_{\frac{\pi - \theta}{2\pi\alpha}} M_{\frac{1}{2}} \phi_m (x) \right| \le \bigO(\gamma^2).
\end{equation}
If we replace $\bigO(\gamma^2)$ in~\eqref{infbound} and \eqref{infmainbound2}  by $\bigO(\gamma)$, then the range for $m$ increases from $N=\bigO(1)$
to $N=\bigO(\sqrt{1/\gamma})$.
\end{corollary}

\begin{proof}
As in the proof of Theorem \ref{th:main}, we prove the result for even $m$ and $\theta = 0$. Indeed, the proof proceeds exactly as before, except that in this case we no longer need the assumption \eqref{condc}. Thus, \eqref{infbound} holds for $0 < \gamma < 1$ and for all $x \in \mathbb{Z}$. We apply the infinite-dimensional version
of  Proposition~\ref{le:negativeeigenvalues}  to obtain~\eqref{infmainbound2}. 
\end{proof}

\section{Numerical simulations}
\label{s:numerics}

In this section we illustrate our theoretical findings with numerical examples. We consider $\Hopnt$ with the parameters
 $n=10000, \alpha = \frac{1}{n}, \beta=1,\theta =0$, hence in this case $\gamma = \frac{\pi}{n}$.
Figure~\ref{fig1} shows the true eigenvalues of $\Hopn$ and the approximate eigenvalues given by formula~\eqref{eigenbound}.
Depicted are the 600 algebraically largest eigenvalues and their approximations.

\begin{figure}
\begin{center}
\includegraphics[width=100mm]{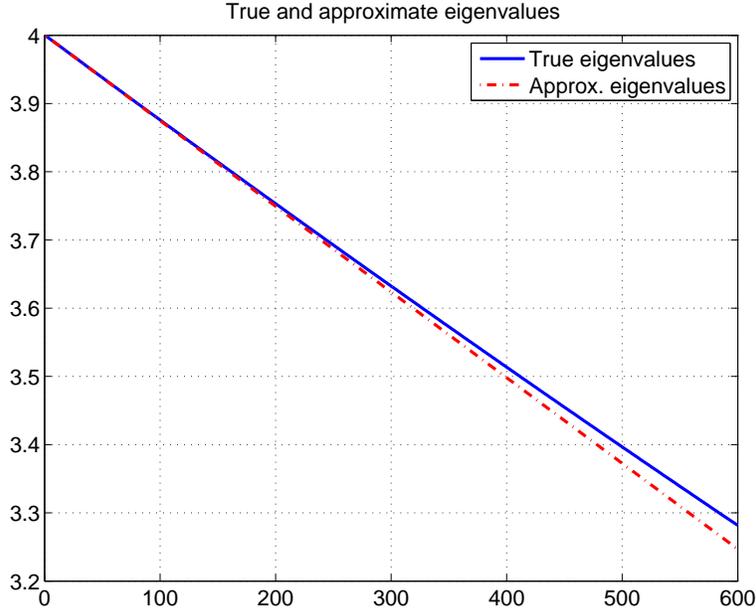}
\caption{True eigenvalues of $\Hopnt$ and their approximation via~\eqref{mainbound}.
The values for $\Hopnt$ are $n=10000, \alpha = \frac{1}{n}, \beta=1,\theta =0$. }
\label{fig1}
\end{center}
\end{figure}

Theorem~\ref{th:main} predicts that the approximation error will be of size 
$\gamma^2$ for the first few approximate eigenvalues, while about $\sqrt{n}$ many approximate eigenvalues deviate from the true eigenvalues
by an error of less than about $1/n$.  Figure~\ref{fig3} confirms this prediction. 
In Figure~\ref{fig3} we show the number of approximate eigenvalues that approximate the true eigenvalues of $\Hopnt$ within an error
of $\gamma$ (red solid line) as well as within an error of $\gamma^2$ (green dash-dotted line). Also depicted is the graph of 
$1/\sqrt{\gamma}$ (blue dashed line). 
The picture confirms clearly our theoretical findings, which imply that the number of approximate eigenvalues that deviate from the true eigenvalues by
an error less than $\gamma^2$ is of size $\bigO(1)$, i.e, independent of $n$, while the number of approximate eigenvalues that deviate from the 
true eigenvalues by an error less than $\gamma$ scales like $1/\sqrt{\gamma}$, thus in this example like $\sqrt{\frac{n}{\pi}}$.

\begin{figure}
\begin{center}
\includegraphics[width=100mm]{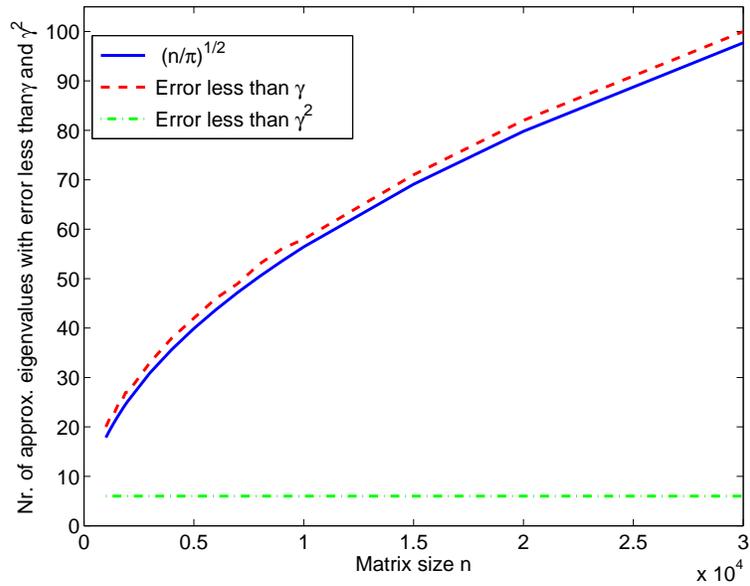}
\caption{Shown are the number of approximate eigenvalues that approximate the true eigenvalues of $\Hopnt$ within an error
of $\gamma$ (red solid line) as well as within an error of $\gamma^2$ (green dash-dotted line). Also depicted is the graph of 
$1/\sqrt{\gamma}$ (blue dashed line). The values for $\Hopnt$ are $n=1000,\dots, 30000, \alpha = \frac{1}{n}, \beta=1,\theta =0$, 
hence in this example $\gamma = \frac{\pi}{n}.$
}
\label{fig3}
\end{center}
\end{figure}

\section*{Acknowledgements}
T.S.\ wants to thank Persi Diaconis for introducing him to the topic of this paper.
This research was partially supported by the NSF via grant DTRA-DMS 1042939.


\frenchspacing
\bibliographystyle{plain}
\bibliography{mathbook,mathieubib}



\end{document}